\documentclass{article}
\usepackage{amsmath,amsfonts,amssymb,amsthm}
\usepackage[french,british]{babel}
\usepackage{color}
\usepackage{hyperref}

\providecommand{\R}{\mathbb{R}}

\renewcommand{\leq}{\leqslant}
\renewcommand{\geq}{\geqslant}

\renewcommand{\div}{\operatorname{div}}
\newcommand{\curl}{\operatorname{curl}}

\newcommand{\dist}{\operatorname{dist}}
\newcommand{\Id}{\operatorname{Id}}
\newtheorem{Theorem}{Theorem}
\newtheorem{Definition}{Definition}

\newtheorem{Proposition}{Proposition}
\newtheorem{Lemma}{Lemma}

\begin{document}

\date{23rd July 2018}
\title{On the 2D ``viscous incompressible fluid + rigid body'' system with Navier conditions and unbounded energy \\[1ex]
\itshape Sur le mouvement d'un corps rigide dans un \'ecoulement bidimensionel d'un fluide visqueux incompressible avec conditions au bord de Navier et \'enegie infinie}
\author{Marco Bravin\footnote{Institut de Math\'ematiques de Bordeaux, UMR CNRS 5251,
Universit\'e de Bordeaux, 351 cours
de la Lib\'eration, F33405 Talence Cedex, France. }}

\maketitle

\begin{abstract}
In this paper we consider the motion of a rigid body  in a viscous incompressible fluid when some Navier slip conditions are prescribed on the body's boundary.
The whole  ``viscous incompressible fluid + rigid body'' system is assumed to occupy the full plane $\R^{2}$. We prove the existence of global-in-time weak solutions with constant non-zero circulation at infinity.
\end{abstract}

\begin{otherlanguage}{french}
\begin{abstract}
Dans cet article, nous consid\'erons le mouvement d'un corps rigide dans un fluide visqueux incompressible avec des conditions de glissement avec friction de Navier \'a l'interface. Le syst\'eme ``fluide+corps rigide'' est suppos\'e occuper le plan tout entier. Nous prouvons l'existence de solutions globales en temps avec une circulation constante non nulle \'a l'infini. 
\end{abstract}
\end{otherlanguage}

\section*{Introduction}

The problem of well-posedness of Navier-Stokes equations with infinite energy in dimension two has been studied a lot in the past years. We recall the work \cite{ex:mes}, where the authors prove existence for initial data which have measure vorticity and the corresponding uniqueness result is available in \cite{GA:GA}. Other interesting works are \cite{LR} and \cite{MMP}, where the authors prove existence of weak solutions in loc-uniform Lebesque spaces. The first result deals with solution defined in the all space $\mathbb{R}^3$, the second one defined in the half space $ \mathbb{R}^3_{+}$. For exterior domain, where no slip boundary condition are prescribed on the boundary, there exists an existence result for initial data in the weak-$L^2 $ space with some restriction of the concentration of the initial energy. These solutions will remain uniformly bounded in weak-$L^2 $ norm for almost every time and bounded in the $ K_4 $ norm which is the Kato norm for $ p = 4 $.    

In this paper we study weak solutions for viscous incompressible fluid + rigid body system where Navier-type boundary condition are prescribed on the boundary of the solid and the energy is allowed to be infinity for lack of integrability at infinity, more precisely the solutions behaves like a $ x^{\perp}/2\pi |x|^2$ at infinity.

In the case of finite energy, a wide literature is presence for example \cite{PS}, \cite{exi:GeH}, \cite{IO}, in particular in \cite{PS} existence of weak solution is proved. The goals of this work are to extend the definition of weak solutions presented in \cite{PS} in our setting and prove existence. The main contributions are the extension of the definition of weak solutions and the density argument presented in Lemma \ref{lem2}, this last result is also essential to make the proof of Theorem 1 of \cite{PS} correct.

\section{The 2D ``viscous incompressible fluid + rigid body''  system  with Navier conditions}

We study the Cauchy problem for a system describing the motion of a
rigid body immersed in an viscous incompressible fluid when some Navier slip conditions are prescribed on the body's boundary.
In \cite{PS}, the existence of global weak solutions with finite energy to the
Cauchy problem were established, in the case where the whole system occupies the full space
$\R^3$.  Moreover, several properties of these solutions were exhibited.
We consider here the $2$D case, for which our analysis can be carried out for initial data corresponding to unbounded fluid kinetic energy.

Let us therefore consider $\mathcal{S}_0$ a closed, bounded, connected and simply connected subset of the plane with smooth boundary.
We assume that the body initially occupies the domain $\mathcal{S}_0$ and we denote $\mathcal{F}_0 = \R^2  \setminus \mathcal{S}_0 $ the domain occupied by the fluid. 

The equations modelling the dynamics of the system in the body frame read then
\begin{eqnarray}
\label{NS11-2d}
\displaystyle \frac{\partial u}{\partial t}
+ \left[(u-\ell-r x^\perp)\cdot\nabla\right]u
+ r u^\perp +\nabla p = \nu \Delta  u&& x\in \mathcal{F}_{0} ,\\
\label{NS12-2d}
\div u = 0 && x\in \mathcal{F}_{0} , \\
\label{NS13-2d}
u\cdot n = \left(\ell +r x^\perp\right)\cdot n && x\in \partial \mathcal{S}_0, \\
\label{NS14-2d}
(D(u) n )  \cdot \tau = - \alpha (u -\ell-r x^\perp ) \cdot \tau   && \text{for} \ x\in \partial \mathcal{S}_0,  \\
\label{Solide11-2d}
m \ell'(t)= - \int_{\partial \mathcal{S}_0}  \sigma  n \, ds-mr \ell^\perp, & & \\
\label{Solide12-2d}
\mathcal{J} r'(t)= -\int_{\partial \mathcal{S}_0} x^\perp \cdot  \sigma  n \, ds, & &  \\
\label{NS1ci}
u(0,x)= u_0 (x) && x\in \mathcal{F}_{0} ,\\
\label{Solide1ci-2d}
\ell(0)= \ell_0,\ r (0)= r_0,
\end{eqnarray}
where $ u_0\cdot n = (l_0+r_0 x^{\perp})\cdot n $ for any $ x \in \partial \mathcal{S}_0 $. Here $ u =(u_1,u_2)$ and $p$ denote the velocity and pressure fields, $\nu > 0 $ is the viscosity, $n $ and $\tau$ are the unit outwards normal and counterclockwise tangent vectors to the boundary of the fluid domain, $\alpha \geq 0 $ is a material constant (the friction
coefficient).
$m$ and $ \mathcal{J}$ denote respectively the mass and the moment of inertia of the body  while the fluid  is supposed to be  homogeneous of density $1$, to simplify the notations.
The Cauchy stress tensor is defined by $ \sigma = -p \Id_2 + 2 \nu D(u)$, where $D(u) =  ( \frac{1}{2} ( \partial_{j} u_{i} +  \partial_{i} u_{j} ) )_{1 \leqslant i,j \leqslant  2} $ is the deformation tensor.

When $x=(x_1,x_2)$ the notation $x^\perp $ stands for $x^\perp =( -x_2 , x_1 )$,
$l(t)$
is the velocity of the center of mass of the body and $r(t)$ denotes the angular velocity of the rigid body. Finally to shorter the notation we will write $ u_S = l +r x^{\perp}$.

\section{Leray-type solutions with infinite energy}

We are interested in solution with an initial data 
\begin{equation*}
u_0 = \tilde{u}_0+\beta H_{\mathcal{S}_0} \in L^2_{\sigma}(\mathcal{F}_0)\oplus \mathbb{R} H_{\mathcal{S}_0},
\end{equation*}
where $ H_{\mathcal{S}_0} $ is the unique solution vanishing at infinity of 
\begin{gather*}
\div H_{\mathcal{S}_0} = 0 \quad   \text{for}  \ x\in  \mathcal{F}_{0}, \\
\curl H_{\mathcal{S}_0} = 0 \quad   \text{for}  \ x\in  \mathcal{F}_{0}, \\
H_{\mathcal{S}_0} \cdot n = 0 \quad   \text{for}  \ x\in   \partial \mathcal{S}_0, \\
\int_{\partial \mathcal{S}_0 } H_{\mathcal{S}_0} \cdot \tau \, ds = 1 .
\end{gather*}
See, for instance, \cite{kikuchi}. This solution is smooth and decays like $1/ | x| $ at infinity.
For any $ x_0 $ in the interior of $ \mathcal{S}_0 $, we also have 
%
%
\begin{equation} \label{EstiHache}
H_{\mathcal{S}_0}, \nabla H_{\mathcal{S}_0}  \in L^\infty (\mathcal{F}_0) \ \text{ and } \
 H_{\mathcal{S}_0} -  \frac{(x-x_0)^{\perp}}{2 \pi | x-x_0 |^{2}}, \ \nabla H_{\mathcal{S}_0} , \ H_{\mathcal{S}_0}^\perp - (x-x_0)^\perp  \cdot\nabla H_{\mathcal{S}_0} \in L^2 (\mathcal{F}_0),
\end{equation}
but $ H_{\mathcal{S}_0} $ is not a $L^2 $ function.
In the case of regular solutions to the Euler equations this vector field  is useful to take the velocity circulation around the body into account, which is a conserved quantity according to Kelvin's theorem.

First of all we note that $ H_{\mathcal{S}_0} $ satisfies the equations \eqref{NS11-2d}-\eqref{NS13-2d} unless the boundary condition \eqref{NS14-2d}. This leads us to expect that a solution $ u $ of \eqref{NS11-2d}-\eqref{Solide1ci-2d} with initial data $ \tilde{u}_0+\beta H_{\mathcal{S}_0} $ is of the form 
\begin{equation*}
u = \tilde{u}+\beta H_{\mathcal{S}_0}, \text{ with } \tilde{u} \in L^2_{\sigma}(\mathcal{F}_0)
\end{equation*} 
and $ \beta $ is independent of time.

We now introduce a definition of Leray-type solutions for these initial data. First of all in the literature, for example in \cite{PS}, there is already a definition of weak solutions of Leray-type with finite energy, i.e. with $ \beta = 0 $, so we want to be coherent with this definition. In the next subsection we recall the definition of weak solution with finite energy coming from \cite{PS} and then we notice that we can extend this definition in a straight-forward way to our setting.    

\subsection{A weak formulation with finite energy}

Let us now use the notation $ \mathcal{H} $ for  the following space
\begin{equation*}
 \mathcal{H} = \{ \phi \in L^{2} (\R^{2}) \big| \ \div \phi = 0  \text{ in }  \R^{2} \text{ and } D(\phi) = 0  \text{ in }  \mathcal{S}_0 \} .
\end{equation*}
For all  $\phi \in  \mathcal{H}$, there exist $\ell_{\phi} \in \R^{2}$ and $r_{\phi} \in \R $ such that for any $x \in  \mathcal{S}_0$,
$\phi (x) = \ell_{\phi} + r_{\phi} x^\perp $.
Therefore we extend the initial data $v_{0}$  by setting
$v_{0} = \ell_{0} + r_{0}  x^\perp $  for $x \in \mathcal{S}_0 $.
Conversely, when $\phi \in  \mathcal{H}$, we denote  $\phi_{\mathcal{S}}$ its restriction to $\mathcal{S}_0 $.
Now we endow the space $\mathcal{H}$ with the following inner product
\begin{equation*}
(\phi  , \psi  )_{\mathcal{H}} =    \int_{\mathcal{F}_0 }  \phi \cdot  \psi  \, dx  +  m  \ell_{\phi} \cdot  \ell_{\psi} + \mathcal{J}  r_{\phi}   r_{\psi}  .
\end{equation*}
which is equivalent to the restriction of the $ L^2(\mathbb{R}^2) $ inner product to the subspace $\mathcal{H}$. Let us also denote
\begin{gather*}
\underline{\mathcal{V}}=   \left\{  \phi \in   \mathcal{H} \bigg| \ \int_{ \mathcal{F}_0 } | \nabla \phi  (y) |^2  dy < + \infty    \right\} \ \text{  with  norm } \ \ \|  \phi  \|_{\underline{\mathcal{V}}} = \|  \phi  \|_{\mathcal{H}} + \|  \nabla  \phi  \|_{L^{2} (\mathcal{F}_0 ,dy )   } , 
\\ \mathcal{V} =   \left\{  \phi \in   \mathcal{H} \bigg| \ \int_{ \mathcal{F}_0 } | \nabla \phi  (y) |^2 (1 + | y |^2 ) dy < + \infty    \right\}  \ \text{  with  norm } \ \ \|  \phi  \|_{\mathcal{V}}= \|  \phi  \|_{\mathcal{H}} + \|  \nabla  \phi  \|_{L^{2} (\mathcal{F}_0  ,(1 + | y |^2 )^{\frac{1}{2}} dy ) },  \\
\widehat{\mathcal{V}} =   \bigg\{  \phi \in   \mathcal{V} \bigg| \ \phi\vert_{\mathcal{F}_0} \in \text{Lip} (\overline{\mathcal{F}_0} )    \bigg\}  \text{ with norm  } \|  \phi  \|_{\widehat{\mathcal{V}}} = \|  \phi  \|_{\mathcal{V}} + \|   \phi  \|_{ \text{Lip} ( \overline{\mathcal{F}_0})   }.
\end{gather*}
Let us emphasize that $\widehat{\mathcal{V}} \subset \mathcal{V} \subset \underline{\mathcal{V}}$. We define formally for appropriate $u$ and $v$, 
\begin{eqnarray*}
a(u,v) &=& - \alpha  \int_{\partial \mathcal{S}_0} (u - u_\mathcal{S} ) \cdot (v - v_\mathcal{S} ) - \int_{\mathcal{F}_0 }     D(u) :  D(v)
\\ b(u,v,w) &=&
\int_{\mathcal{F}_0 } \Big(     [   ( u -  u_\mathcal{S} )  \cdot\nabla w ] \cdot v   -  r_u  v^\perp \cdot w ) \Big) - m r_u \ell_{v}^\perp \cdot \ell_w  .
\end{eqnarray*}
The next straight-forward proposition clarify in which spaces $ a $ and $ b $ are defined.
\begin{Proposition}\label{pro:fab} The following holds true:
\begin{enumerate}
\item[i.] $b$ is a  trilinear  continuous map from $\underline{\mathcal{V}} \times \underline{\mathcal{V}} \times \mathcal{V}$ to $\R$, i.e. there exists a constant $C>0$ such that
for any $(u,v,w) \in \underline{\mathcal{V}} \times  \underline{\mathcal{V}} \times \mathcal{V}$,
\begin{equation*}
|b(u,v,w)  |  \leqslant C \| u  \|_{\underline{\mathcal{V}}} \,  \| v  \|_{\underline{\mathcal{V}}} \,  \| w \|_{\mathcal{V}} .
\end{equation*} 
Moreover if $ v \in \underline{\mathcal{V}} $ it holds $ b(u,v,v) = 0 $ and if $ v, w \in \mathcal{V} $, it holds $ b(u,v,w) = - b(u,w,v) $.
\item[ii.] $b$ can be extended to a continuous map from $\mathcal{H} \times \mathcal{H} \times \widehat{\mathcal{V}}$ to $ \mathbb{R}$, i.e. there exists a constant $C>0$ such that
for any $(u,v,w) \in \mathcal{H} \times \mathcal{H} \times \widehat{\mathcal{V}}$,
\begin{equation*}
|b(u,v,w)  |  \leqslant C \| u  \|_{\mathcal{H}} \,  \| v  \|_{\mathcal{H}} \,  \| w \|_{\widehat{\mathcal{V}}} . 
\end{equation*} 
\item[iii.] $ a(.,.) $ is a continuous map from $ \underline{\mathcal{V}}\times \underline{\mathcal{V}} $ to $ \R $, i.e. for any $u,v$ in $\underline{\mathcal{V}} $,
\begin{equation*}
| a(u,v)  |  \leq  C   \|  u \|_{\underline{\mathcal{V}}} \,    \|  v \|_{\underline{\mathcal{V}}}   .
\end{equation*}
\end{enumerate}
\end{Proposition}

We are now able to state the definition of weak solution defined in \cite{PS}.

\begin{Definition}
Let $ v_0 \in \mathcal{H}$, we say that $ v \in C(0,T; \mathcal{H}) \cap L^2(0,T;\underline{\mathcal{V}}) $ is a solution of (\ref{NS11-2d})-(\ref{Solide1ci-2d}) with finite energy if and only if for all $ \varphi \in C^{\infty}([0,T];\mathcal{H}) $ such that $ \varphi|_{\mathcal{F}_0} \in C^{\infty}(0,T; C^{\infty}_c(\overline{\mathcal{F}_0})) $ and for a.e. $ t \in [0,T] $ it holds
\begin{equation*}
(v,\varphi)_{\mathcal{H}}  (t) -  (v_{0},\varphi |_{t=0})_{\mathcal{H}}   =
 \int_{0}^{t} \Big[ (v , \partial_{t} \varphi)_{\mathcal{H}}
 + 2  \nu a(v,\varphi)  -  b(v,\varphi , v)    \Big] .
\end{equation*}
\end{Definition}


\subsection{A weak formulation with infinite energy}

To extend the definition of weak solution in the case of unbounded energy we start with noticing that we can continuously extend the map $ a $ and $ b $ in our new setting.
First of all for $ X $ one of the spaces $\mathcal{H}$, $ \underline{\mathcal{V}}$, $ \mathcal{V}$ or $ \widehat{\mathcal{V}}$, the space
$
X\oplus \R H $ is endowed with the norm $ \| u \|_{X\oplus \R H} = \|\tilde{u}+\beta H\|_{X\oplus \R H} = \|\tilde{u}\|_{X}+|\beta|, 
$ moreover we use the convention that $ u_\mathcal{S} = \tilde{u}_\mathcal{S} $, $ l_u = l_{\tilde{u}}$ and $ r_u = r_{\tilde{u}}$, i.e. we extend  the function $ H $ by $ 0 $ inside the solid $\mathcal{S}_0 $.

\begin{Proposition} \label{ext:ab} The map $ a $ and $ b $ can be linearly extended as follow:
\begin{enumerate}
\item[i.] the map $b$  can be continuously extended to a trilinear map on $(\underline{\mathcal{V}} \oplus \R  H) \times  \underline{\mathcal{V}} \times (\mathcal{V} \oplus \R  H) $ by
\begin{equation*}
b(u,\tilde{v},w) =
\int_{\mathcal{F}_0 } \Big(     [   ( u -  u_\mathcal{S} )  \cdot\nabla w ] \cdot \tilde{v}   -  r_u  \tilde{v}^\perp \cdot w ) \Big) - m r_u \ell_{\tilde{v}}^\perp \cdot \ell_w.
\end{equation*}
The continuity assumption is equivalent to the following inequality :
there exists a constant $C>0$ such that
for any  $(u = \tilde{u} + \beta_1 H  , \tilde{v} , w=  \tilde{w}+ \beta_3 H)  \in (\underline{\mathcal{V}}\oplus \R  H) \times  \underline{\mathcal{V}} \times (\mathcal{V} \oplus \R  H)$,
\begin{equation*}
|b(u, \tilde{v},w)  |  \leqslant C ( \| \tilde{u} \|_{\underline{\mathcal{V}}} + |   \beta_1 |) \,  \|  \tilde{v}  \|_{\underline{\mathcal{V}}} \, ( \| \tilde{w} \|_{\mathcal{V}} +  |   \beta_3 |).
\end{equation*}
\item[ii.] The map $ b(H,.,.) $, $ b(.,.,H) $ are continuous bilinear map from $ \underline{\mathcal{V}}\times\underline{\mathcal{V}}$ to $ \mathbb{R}$ and by Blasius lemma $ b(H,\tilde{v},H) = 0 $ for any $ \tilde{v} \in \underline{\mathcal{V}}$ (we refer to \cite[Lemma A.1]{GLS} for the Blasius lemma).
\item[iii.] For $ u \in \underline{\mathcal{V}}\oplus \R H$ and $\tilde{v} \in  \underline{\mathcal{V}} $, we have $ b(u,\tilde{v},\tilde{v}) = 0 $. Moreover if $ \tilde{v}, \tilde{w} \in \mathcal{V} $, it holds $ b(u, \tilde{v}, \tilde{w}) = - b(u,\tilde{w},\tilde{v}) $.
\item[iv.] The trilinear map $ b $ can be extended in a unique way on  $(\mathcal{H} \oplus \R  H) \times \mathcal{H}\times  (\widehat{\mathcal{V}}  \oplus \R  H)$ in a continuous way, i.e. there exists a constant $C>0$ such that
for any $(u,\tilde{v},w) = (\tilde{u} + \beta_1 H,  \tilde{v} , \tilde{w}+ \beta_3 H)$,
\begin{equation*}
|b(u,\tilde{v},w)  |  \leqslant C (\| \tilde{u}  \|_{\mathcal{H}}+ |   \beta_1 |) \,  \| \tilde{v}  \|_{\mathcal{H}} \, ( \| \tilde{w} \|_{\widehat{\mathcal{V}}} + |   \beta_3 |).
\end{equation*}
\item[v.] $ a(.,.) $ can be extended to a continuous bilinear map from $ (\underline{\mathcal{V}} \oplus \R  H )\times \underline{\mathcal{V}} $ to $ \R $, where for any $ (u,\tilde{v}) $
\begin{equation*}
a(u,\tilde{v}) = - \alpha  \int_{\partial \mathcal{S}_0} (u - u_\mathcal{S} ) \cdot (\tilde{v} - \tilde{v}_\mathcal{S} ) - \int_{\mathcal{F}_0 }     D(u) :  D(\tilde{v}).
\end{equation*}
\end{enumerate}
\end{Proposition}
\begin{proof} Point \textit{i.} is direct consequence of point \textit{ii.} so we begin by \textit{ii.} For $ (\tilde{u}, \tilde{v}) \in \underline{\mathcal{V}}\times \underline{\mathcal{V}} $,
\begin{equation*}
 b(\tilde{u},\tilde{v},H) =
\int_{\mathcal{F}_0 }  [   \tilde{u}  \cdot \nabla H ] \cdot \tilde{v}
-  \int_{\mathcal{F}_0 } [ \ell_{\tilde{u}}  \cdot \nabla H ]  \cdot \tilde{v}
- r_{\tilde{u}} \int_{\mathcal{F}_0 } (x^\perp \cdot  \nabla  H -H^\perp ) \cdot \tilde{v} 
\end{equation*}
is well defined thanks to  \eqref{EstiHache}, moreover there exists $ C > 0$ such that $
|b(\tilde{u},\tilde{v}, H)| \leq C\|\tilde{u}\|_{\underline{\mathcal{V}}}\|\tilde{v}\|_{\underline{\mathcal{V}}} $.
For $ (\tilde{v}, \tilde{w}) \in  \underline{\mathcal{V}}  \times  \underline{\mathcal{V}}$,
$ b(H,\tilde{v}, \tilde{w} ) = \int_{\mathcal{F}_0 }[ H  \cdot\nabla] \tilde{w}  \cdot \tilde{v}$.
Thanks to \eqref{EstiHache}, it is clear that
$  |b(H,\tilde{v}, \tilde{w} )| \leq C \|\tilde{v}\|_{\underline{\mathcal{V}}}\|\tilde{w}\|_{\underline{\mathcal{V}}}$.
Moreover, using Blasius lemma, we have for any $\tilde{v} \in \underline{\mathcal{V}}$, $
 b(H,\tilde{v},H) = 0 $.
This concludes the proof of point \textit{i.} and \textit{ii.}.
For \textit{iii.}, we use an integration by parts to see that for any $ \tilde{v} \in \underline{\mathcal{V}}$ we have
$
b(H,\tilde{v},\tilde{v}) = 0$,
which implies, together with point \textit{i.} of Proposition \ref{pro:fab}, that  it holds $
b(u,\tilde{v},\tilde{v}) = 0 $ for any $ u = \tilde{u} + \beta H \in \underline{\mathcal{V}} \oplus \R H $.
Integrating by part we have also that for any  $u \in \underline{\mathcal{V}} \oplus \R  H$, for any  $\tilde{v}, \tilde{w} \in {\mathcal{V}}$,
$
b(u, \tilde{v}, \tilde{w}) = - b(u,\tilde{w} , \tilde{v}) $.

Point \textit{iv.} is trivial after notice that $ \nabla H \in L^{\infty} $ and recall \textit{ii.} of Proposition \ref{pro:fab}.

Finally to prove \textit{v.} we use the same procedure of point \textit{iii.} of Proposition \ref{pro:fab}. 
\end{proof}

We now introduce the definition of weak solution, with possibly unbounded energy,  of the system \eqref{NS11-2d}-\eqref{Solide1ci-2d}.
\begin{Definition}
 Let $u_{0} = \tilde{u}_0 + \beta H \in  \mathcal{H} \oplus \R  H$ and  $T >0$. We say that
$u = \tilde{u} + \beta H $ where $$ \tilde{u} \in   C ( [0,T ]  ;  \mathcal{H}) \cap  L^2 ( 0,T  ; \underline{\mathcal{V}})$$
 is a weak solution of \eqref{NS11-2d}-\eqref{Solide1ci-2d} if for any test function $ \varphi \in C^1([0,T];\mathcal{H})$ such that $ \varphi
|_{\overline{\mathcal{F}_0}} \in C^{1} ([0,T ] ;  C^{\infty }_{c} (\overline{\mathcal{F}_0} ))$
\begin{equation*}
(\tilde{u},\varphi)_{\mathcal{H}}  (t) -  (\tilde{u}_{0},\varphi |_{t=0})_{\mathcal{H}}   =
 \int_{0}^{t} \Big[ (\tilde{u} , \partial_{t} \varphi)_{\mathcal{H}}
 + 2  \nu a(u,\varphi)  -  b(u,\varphi , u)    \Big] .
\end{equation*}
\end{Definition}
Observe that we took into account here that $H$ is time independent, and $\beta$ as well. For our convenience we give an equivalent but more explicit definition of weak formulation of the system \eqref{NS11-2d}-\eqref{Solide1ci-2d}.  
\begin{Definition}[Weak solution with $\beta$ circulation at infinity] 
 Let $ \tilde{u}_0 \in \mathcal{H} $ and $ T > 0 $ given. We say that 
\begin{equation*}
\tilde{u} \in C([0,T];\mathcal{H} ) \cap L^2((0,T);\underline{\mathcal{V}})
\end{equation*}
is a weak solution for 2D Navier-Stokes with $\beta $ circulation at infinity if for every test function $ \varphi \in C^1([0,T];\mathcal{H}) $ with $\varphi|_{\overline{\mathcal{F}_0}} \in C^{1}([0,T];C^{\infty}_c(\overline{\mathcal{F}_0}))    $, it holds
\begin{align*}
(\tilde{u}(t),\varphi(t))_{\mathcal{H}}-(\tilde{u}_0,\varphi(0))_{\mathcal{H}}= \int_0^t \Big[&(\tilde{u},\partial_t \varphi)_{\mathcal{H}}+2\nu a(\tilde{u},\varphi)+2\beta \nu a(H,\varphi)  \\ & -b(\tilde{u},\varphi,\tilde{u})-\beta b(H,\varphi,\tilde{u})-\beta b(\tilde{u},\varphi, H) \Big]dt.
\end{align*}
\end{Definition}

To conclude this section, we observe that any smooth solution of \eqref{NS11-2d}-\eqref{Solide1ci-2d} with infinite energy is also a weak solution.
\begin{Proposition}
Let $ u = \tilde{u} + \beta H $ a smooth solution of \eqref{NS11-2d}-\eqref{Solide1ci-2d} with initial data $ u_0 = \tilde{u}_0 + \beta H $, then $ \tilde{u} $ is a weak solution for 2D Navier-Stokes with $ \beta $ circulation at infinity.  
\end{Proposition}
\begin{proof}
Multiply the equation \eqref{NS11-2d} by the test function $ \varphi $, integrate in all $\mathcal{F}_0$, integrate by parts and use the boundary condition.
\end{proof}

\section{Result}

The following result establishes the existence of global weak solutions of the system \eqref{NS11-2d}-\eqref{Solide1ci-2d}.
\begin{Theorem}
\label{2dNSBodyWeak}
Let $ \tilde{u}_0 \in \mathcal{H} $ and let $ T > 0 $. Then there exists a weak solution $ \tilde{u} \in \mathcal{H} $ of 2D Navier-Stokes with $\beta$ circulation at infinity in $ C([0,T]; \mathcal{H})\cap L^{2}(0,T;\underline{\mathcal{V}}) $ such that satisfies the following energy inequality: for almost every $ t \in [0,T] $ we have
\begin{equation}
\label{NSBodyWeakEnergy}
\frac{1}{2}\|\tilde{u}(t,.)\|^2_{\mathcal{H}}+2\nu\int_{(0,t)\times\mathcal{F}_0}  |D(\tilde{u})|^2 +  2\alpha \nu \int_0^t\int_{\partial \mathcal{S}_0}|\tilde{u}-\tilde{u}_{\mathcal{S}}|^2 \leq \nonumber  C(1+\|\tilde{u}_0\|^2_{\mathcal{H}}),
\end{equation}   
where $ C $ depends on $ T, \mathcal{S}_0, \beta$ and $ \nu $. 
%
%
Moreover $(l,r) \in H^1(0,T; \R^2 \times \R)$.
\end{Theorem}
The motivation that drives us to study this special infinite-energy solutions is to study the ``inviscid+shrinking-body'' limit. For the inviscid limit we recall the result from \cite{PS}, where the authors proved that as $  \nu $ goes to zero, the solutions $ u_{\nu} $ converge to the solution of the corresponding Euler system. In \cite{PS}, the ``rigid+body'' system occupies all the space $ \mathbb{R}^3$, in the case of $ \mathbb{R}^2 $ the situation is a bit more tricky and the argument of \cite{PS} holds at least in the case the solid is a disk. Moreover by the work \cite{GLS}, we know that as the size of the object goes to zero (and the mass remains constant) the system converges to a variant of the vortex-wave system where the vortex, placed in the point occupied by the shrunk body is accelerated by a Kutta-Joukowski-type lift force. In the massless case, i.e. the density of the object is constant respect to the scale of the object, similar results are available when the fluid satisfies incompressible Navier-Stokes equation  and no-slip boundary conditions are prescribe on the boundary of the solid, for example in \cite{LT} it is proven that for a fixed viscosity the ``fluid+disk'' system converges to the Navier-Stokes system in all $\mathbb{R}^2 $ when the object shrinks to a point. The goal of further studies is to understand the limiting equations when both the viscosity and the size of the object go to zero at the same time (in both mass and massless cases) and to find in the limit a similar system of the one in \cite{GLS}. We expect that the appearance of a Kutta-Joukowski-type lift force in the limiting system is strictly related to the presence of the circulation due to $ \beta H$, i.e. in the absence of this term we do not expect to see any force on the point mass in the limit. Indeed in the case where the vorticity is integrable, $ \beta H $ denotes the circulation at infinity.  

Before moving to the proof of the theorem we present two density results. For the first one we do not claim originality but we were not able to find a reference in the literature. The second result is one of the main contribution of the paper. Lemma \ref{lem2} is also essential in \cite{PS}, where we propose to change the set $ \mathcal{T} = \{ \varphi \in C^{\infty}_{c, \sigma}(\mathbb{R}^2) | D(\varphi) = 0 \text{ in } \mathcal{S}_0 \}$ with the set defined in 
(\ref{Y}) in the proof of Theorem 1. The set $ \mathcal{T}$ is not dense in $ \underline{\mathcal{V}} $ neither in $\mathcal{V} $. On the other hand we will introduce below, cf. (\ref{Y}), a set $ \mathcal{Y}$ which is dense and has all the property to make the proof of Theorem 1 of \cite{PS} working. To see that $ \mathcal{T} $ is not dense in $\mathcal{V}$, it is enough to consider $ \mathcal{S}_0 = B_1(0) $ and the function 
\begin{equation*}
f(x) = \begin{cases} 0 & \quad \text{ in } B_1(0), \\  \nabla^{\perp} ( x^2 \chi ) & \quad \text{ elsewhere, }   \end{cases}
\end{equation*}
where $ \chi $ is a smooth cut off such that $ \chi \equiv 1 $ in $ B_2(0) $ and $ \chi \equiv 0 $ outside $ B_4(0) $. It is clear that $ f \in \mathcal{V} \subset \underline{\mathcal{V}} $. Suppose by contradiction that there exist approximations $ f_{\varepsilon} \in \mathcal{T} $ such that $ f_{\varepsilon} \to f $ in $ \mathcal{V} $, then $ l_{f_{\varepsilon}} \to l_{f} = 0 $  and $ r_{f_{\varepsilon}} \to r_{f} = 0 $. $f_{\varepsilon}|_{\mathbb{R}^2 \setminus B_1(0) } \to f|_{\mathbb{R}^2 \setminus B_1(0) } $ in $ H^{1}(\mathbb{R}^2 \setminus B_1(0))$, then by trace theorem $  f_{\varepsilon}|_{\partial B_1(0)} \rightarrow f|_{\partial B_1(0)}  $ and  $ f_{\varepsilon}|_{\partial B_1(0)} \to 0 $ in $ L^2(\partial B_1(0)) $ but $ f|_{\partial B_1(0)} = 2 x^{\perp}$ which is a contradiction.     

We start by presenting the first density result.

\begin{Lemma}
\label{den:BOOM}
Let $ \Omega $ an open, bounded  subset of $\mathbb{R}^2 $ with smooth boundary such that $\partial \Omega = \cup \Gamma_i $ where $ \Gamma_i $ for $ i = 0,\dots, n$ are open connected components of the boundary with $ \Gamma_i\cap\Gamma_j = \emptyset $ for $ i\neq j $, then the set $ C^{\infty}_{\sigma}(\Omega) \cap L^2_{\sigma}(\Omega) $ of smooth divergence-free functions with $ 0 $ normal component on the boundary $ \partial \Omega $ is dense in $ H^1(\Omega) \cap L^2_{\sigma}(\Omega)$. 
\end{Lemma}

\begin{proof}
Let $ v \in H^1(\Omega) \cap L^2_{\sigma}(\Omega) $, then by \cite[Corollary 3.3]{GIR:RAV} there exists a stream function $ \psi $ such that $ \nabla^{\perp} \psi = v $ and $ \psi \in H^{2}(\Omega) $. Using the condition $ v \cdot n  = 0 $ on $ \partial\Omega $, $ \psi $ satisfies w.l.o.g. 

\begin{equation*}
\begin{cases}
-\Delta \psi = - \curl v \quad & \text{ in } \Omega, \\
\psi = 0 \quad & \text{ on } \Gamma_0, \\
\psi = c_i \quad & \text{ on } \Gamma_i,
\end{cases}
\end{equation*}
for some constant $ c_i$. Consider $ \eta_{\varepsilon} $ a symmetric convolution kernel of mass $ 1 $ with support in $ B_{\varepsilon}(0)$ and consider $ \chi_{\varepsilon} $ the characteristic function such that $ \chi_{\varepsilon}(x) = 1 $ if $ \dist(x,\partial\Omega) > \varepsilon$ and $ 0 $ else. We define 

\begin{equation*}
\begin{cases}
-\Delta \psi_{\varepsilon} = - \left(\chi_{3\varepsilon}\curl v\right)*\eta_{\varepsilon} \quad & \text{ in } \Omega, \\
\psi_{\varepsilon} = 0 \quad & \text{ on } \Gamma_0, \\
\psi_{\varepsilon} = c_i \quad & \text{ on } \Gamma_i.
\end{cases}
\end{equation*}

The functions $ v_{\varepsilon} = \nabla^{\perp} \psi_{\varepsilon} $ are the desired approximations of $ v $. First of all we prove that $ v_{\varepsilon} \in C^{\infty}_{c}(\Omega) \cap L^{2}_{\sigma}(\Omega) $. This is clear by elliptic regularity and $ v_{\varepsilon} \cdot n = \nabla^{\perp} \psi_{\varepsilon}\cdot n = \nabla \psi \cdot \tau = 0 $ on $ \partial \Omega $ ($\psi_{\varepsilon}$ is constant in any $ \Gamma_i$). To prove the convergence we notice, by elliptic regularity

\begin{equation*}
\|v_{\varepsilon} - v\|_{H^1(\Omega)} \leq \|\psi_{\varepsilon} - \psi\|_{H^2(\Omega)} \leq C \|\left(\chi_{3\varepsilon}\curl v\right)*\eta_{\varepsilon} - \curl v \|_{L^2(\Omega)} \to 0.  
\end{equation*}

\end{proof}

\begin{Lemma}
\label{lem2}
The set 
\begin{equation}
\label{Y}
\mathcal{Y} = \big\{ v \in L^{2}_{\sigma}(\mathbb{R}^2) \big| \text{ there exist } v_F \in C^{\infty}_{\sigma, c}(\mathbb{R}^2) \text{ and } v_R\in\mathcal{R} \text{ such that } v|_{\mathcal{F}} =  v_F|_{\mathcal{F}} \text{ and }  v|_{\mathcal{S}_0} =  v_S|_{\mathcal{S}_0} \big\},
\end{equation}
is dense in $\mathcal{V}$, $\underline{\mathcal{V}}$ and $\mathcal{H} $.
\end{Lemma}

\begin{proof}
The proof in the case of $\mathcal{H} $ is easy. We turn to the case of $\mathcal{V}$ and $\underline{\mathcal{V}}$. The difference between the two spaces is the integrability at $ +\infty $ but this will not change much the proof so we will do it only for $\underline{\mathcal{V}}$. 

Let $ v \in \underline{\mathcal{V}} $ and let $ l $ and $ r $ such that $ v_R = l+x^{\perp}r $. For $ \rho > 0 $ such that $ \rho > \text{diam}(\mathcal{S}_0)$, we define $ \chi_{\rho} $ to be a smooth cut off function such that $0 \leq \chi_{\rho} \leq 1 $, $ \chi_{\rho} = 1 $ in $ B_{\rho}(0) $, $ \chi_{\rho} = 0 $ outside $ B_{2\rho}(0) $ and $|\nabla \chi_{\rho}| \leq C/{\rho} $. Fix $ R > 0 $ such that $ R/4 > \text{diam}(\mathcal{S}_0)$, we decompose $ v = u + v_1 $, where $ u = \nabla^{\perp}(\chi_{R/4}(-l^{\perp}\cdot x + r/2|x|^2))$. The function $ u \in C^{\infty}_{\sigma,c}(\mathbb{R}^2) $ and $ v_1|_{\mathcal{F}} \in H^{1}(\mathbb{R}^2)\cap L^{2}_{\sigma}(\mathcal{F}) $ and $v_1|_{\mathcal{S}_0} = 0 $. By Theorem 3.3 of \cite{GIR:RAV} there exists $ \varphi \in H^2(B_{2R}(0)\setminus \mathcal{S}_0 )$ such that $ v_1 = \nabla^{\perp} \varphi $. We decompose $ v_1 = w + z $ where $ w = \nabla^{\perp}(\chi_R \varphi) $. The function $ z $ is such that $ z|_{B_R(0)} = 0 $ and  $ z|_{\mathbb{R}^2\setminus B_R(0)} \in H^{1}_0(\mathbb{R}^2\setminus \overline{B_R(0)})\cap L^{2}_{\sigma}(\mathbb{R}^2\setminus B_R(0)) = E$, where
$$ E= \overline{\{C^{\infty}_{\sigma,c}(\mathbb{R}^2\setminus \overline{B_R(0)}) \}}^{\|.\|_{H^1}},  $$    
see for example \cite[Section III.4.2]{Gal}. This provides the existence of a sequence $ \tilde{z}_{\varepsilon} \in C^{\infty}_{\sigma,c}(\mathbb{R}^2\setminus \overline{B_R(0)}) $ such that $ \tilde{z}_{\varepsilon} \to z|_{\mathbb{R}^2\setminus B_R(0)}  $ in $ H^1({\mathbb{R}^2\setminus B_R(0)} )$. Let $ z_{\varepsilon} $ to be the extension  by $ 0 $ of $ \tilde{z}_{\varepsilon} $ inside $B_R(0)$, then $ z_{\varepsilon} \to z $ in $ \underline{ \mathcal{V} }$. We now study $ w $. The function $ w \in H^1(B_{4R}(0)\setminus \overline{\mathcal{S}_0})\cap L^2_{\sigma}(B_{4R}(0)\setminus \overline{\mathcal{S}_0})$. By Lemma \ref{den:BOOM} there exist $ \tilde{w}_{\varepsilon} \in C^{\infty}(B_{4R}(0)\setminus \mathcal{S}_0) \cap L^{2}_{\sigma}((B_{4R}(0)\setminus \overline{\mathcal{S}_0}) $ such that $ \tilde{w}_{\varepsilon} \to w|_{B_{4R}(0)\setminus \overline{\mathcal{S}_0}} $ in $H^1(B_{4R}(0)\setminus \overline{\mathcal{S}_0}) $. Let $ \psi_{\varepsilon}\in H^2(B_{4R}(0)\setminus \overline{\mathcal{S}_0}) $ such that $ \tilde{w}_{\varepsilon} = \nabla^{\perp} \psi_{\varepsilon}$. The function $ \psi_{\varepsilon}$ is unique up to a constant, so we choose the unique $ \psi_{\varepsilon} $ such that $ \int_{B_{4R}(0)\setminus \overline{B_{2R}(0)}} \psi_{\varepsilon} = 0 $. Define $ \bar{w}_{\varepsilon} = \nabla^{\perp}(\chi_{2R} \psi_{\varepsilon})$ and denote by $ \bar{w} = w|_{B_{4R}(0)\setminus \overline{\mathcal{S}_0}}$. We have
\begin{align*}
\|\bar{w}-\bar{w}_{\varepsilon}\|_{H^1(B_{4R}(0)\setminus \overline{\mathcal{S}_0})} \leq & \|\bar{w}-\tilde{w}_{\varepsilon}\|_{H^1(B_{4R}(0)\setminus \overline{\mathcal{S}_0})}+C\|\tilde{w}_{\varepsilon}\|_{H^1(B_{4R}(0)\setminus B_{2R}(0))}+\|(\nabla^{\perp}\chi_{2R})\psi_{\varepsilon}\|_{H^1(B_{4R}(0)\setminus \overline{\mathcal{S}_0})} \\
\leq & o(\varepsilon)+ C\|\tilde{w}_{\varepsilon}\|_{H^1(B_{4R}(0)\setminus B_{2R}(0))} + C\|\psi_{\varepsilon}\|_{L^2(B_{4R}(0)\setminus B_{2R}(0)))} \\
\leq & o(\varepsilon)+C\|\tilde{w}_{\varepsilon}\|_{H^1(B_{4R}(0)\setminus B_{2R}(0))} = o(\varepsilon), 
\end{align*}
in fact we can use the Poincar\'e inequality on the $ \psi_{\varepsilon} $ and $ \|\tilde{w}_{\varepsilon}\|_{H^1(B_{4R}(0)\setminus B_{2R}(0))} = o(\varepsilon)$ because $ w = 0 $ in $ B_{4R}(0)\setminus B_{2R}(0)$ ($C$ is a constant that change from line to line). Let $ w_{\varepsilon} $ be the extension by $ 0 $ of $ \bar{w}_{\varepsilon} $. The functions
$$ v_{\varepsilon} = u + w_{\varepsilon}+z_{\varepsilon} \to v \quad \text{ in } \underline{ \mathcal{V }}, $$
Moreover $ v_{\varepsilon}$, $u$, $ w_{\varepsilon}$ and $z_{\varepsilon}$ are element of $ \mathcal{Y}$ (to extend $ w_{F,\varepsilon} $ in the interior is enough to extend $ \psi_{\varepsilon}$).
\end{proof}
Finally we move to the proof of Theorem \ref{2dNSBodyWeak}.

\begin{proof}
The proof of this theorem follows the proof of Theorem 1 in \cite{PS}. The main difficulty is to deal with the fact that the function $ H $ is not an $ L^2{(\mathcal{F}_0)}$ function. In this work we emphasize only the changes in the proof in \cite{PS}, for this reason we divide the proof in several steps as in the paper mention above.   

The idea of the proof is to use an energy estimate to prove that the Galerkin approximation converges. To get  the energy estimate at a formal level is enough to test the equation with $ \tilde{u} $, but this does not work because $ b $ is not bounded in $ \underline{\mathcal{V}}\times\underline{\mathcal{V}}\times\underline{\mathcal{V}} $ but only in $ \underline{\mathcal{V}}\times \underline{\mathcal{V}} \times \mathcal{V}$. The idea is to use a truncation of the solid velocity far from the solid. This procedure was introduced by \cite{orr:tak} in a slightly different setting.

For simplicity in the proof we consider the case $ \beta = 1 $. Dealing with $ \beta \neq 1 $ is not an issue.

\bigskip 

{\bf Truncation.} As said in the beginning we refer to \cite{PS} for more details. Let $ R_0 $ such that $ \mathcal{S}_0 \subset B(0, R_0/2) $. For $ R > R_0 $, let $ \chi_{R}: \mathbb{R}^2 \to \mathbb{R}^2 $ the map such that 
\begin{equation*}
\chi_{R}(x) = \begin{cases}
\chi_{R}(x) = x^{\perp} \quad & \text{ for } x \text{ in } B(0,R/2); \\
\chi_{R}(x) = \frac{R}{|x|}x^{\perp} \quad & \text{ for } x \text{ in } \mathbb{R}^2\setminus B(0,R/2).
\end{cases}
\end{equation*} 
Note that for $ w \in \mathcal{V} $ we have that
\begin{equation*}
\chi_{R}\cdot \nabla w \to x^{\perp}\cdot \nabla w \quad \text{ in } L^2(\mathbb{R}^2) \text{ as } R \to +\infty.
\end{equation*} 

We can use the functions $ \chi_{R} $ to truncate the solid velocity in the following way: we define 
\begin{equation*}
u_{\mathcal{S},R}(t,x) = l(t)+r(t)\chi_{R}(x),
\end{equation*}
and the forms
\begin{equation*}
b_{R}(u,v,w) = m r_u l_u^{\perp}\cdot l_v^{\perp}+\mathcal{J}_0r_ur_vr_w+\int_{\mathcal{F}_0}\left[((u-u_{\mathcal{S},R})\cdot\nabla)w\right]\cdot v - r_uv^{\perp}\cdot w dx. 
\end{equation*}
The advantage of $ b_R $ is that it is a continuous form from $ \underline{\mathcal{V}}\times \underline{\mathcal{V}}\times \underline{\mathcal{V}} $ to $ \mathbb{R}$. Moreover there exists a constant $ C $ independent from $ R $ such that for any $ (u,v,w) \in \underline{\mathcal{V}}\times \underline{\mathcal{V}}\times \mathcal{V} $,
$
|b_{R}(u,v,w)| \leq C \|u\|_{\underline{\mathcal{V}}}\|v\|_{\underline{\mathcal{V}}}\|w\|_{\mathcal{V}}
$
and for any $(u,v) \in \underline{\mathcal{V}}\times \mathcal{V} $, 
$
|b_{R}(u,u,v)|\leq C(\|u\|_{L^4(\mathcal{F}_0)}^2+\|u\|_{\mathcal{H}}^2)\|v\|_{\mathcal{V}}
$.
The cancellation property still hold, in fact for any $(u,v) \in \underline{\mathcal{V}}\times \underline{\mathcal{V}} $,
$
b_{R}(u,v,v) = 0
$.  
Finally we note that for any $ (u,v,w) \in \underline{\mathcal{V}}\times \underline{\mathcal{V}}\times \mathcal{V} $
$b_{R}(u,v,w) \to b(u,v,w) $ when $ R $ goes $ +\infty$.

{\bf Existence for the truncated system.} In this step we present the existence of a solution for the truncated system. We claim that for any $ \tilde{u}_0 \in \mathcal{H} $ and $ T > 0 $, there exists $ \tilde{u}_{R} \in C([0,T];\mathcal{H})\cap L^2([0,T];\underline{\mathcal{V}}) $ such that for all $ \varphi\in C^{\infty}([0,T];\mathcal{H}) $ and $ \varphi|_{\overline{\mathcal{F}_0}} \in C^1([0,T];C^{\infty}_c(\overline{\mathcal{F}_0})) $, and for all $ t \in [0,T] $, it holds
\begin{align*}
(\tilde{u}_R(t),\varphi(t))_{\mathcal{H}}-(\tilde{u}_{R,0},\varphi(0))_{\mathcal{H}}= \int_0^t \Big[&(\tilde{u}_R,\partial_t \varphi)_{\mathcal{H}}+2\nu a(\tilde{u}_R,\varphi)+2\nu a(H,\varphi) \\ & -b_R(\tilde{u}_R,\varphi,\tilde{u}_R)-b(H,\varphi,\tilde{u}_R)-b(\tilde{u}_R,\varphi, H) \Big]dt.
\end{align*}
Moreover $ \tilde{u}_{R} $ satisfies for almost every $ t \in [0,T] $ the energy inequality
\begin{equation*}
\frac{1}{2}\|\tilde{u}_{R}(t)\|_{\mathcal{H}}^2+\int_0^{T}\int_{\partial\mathcal{S}_0}|\tilde{u}_R-\tilde{u}_{R,\mathcal{S}}|^2 dsdt + \int_0^t\int_{\mathcal{F}_0}|D(\tilde{u}_R)|^2dxdt 
\leq C\int_0^T\left(\|\tilde{u}_R\|_{\mathcal{H}}^2+1\right)dt.
\end{equation*}
The idea of the proof is based on the Galerkin method. We consider the set
\begin{equation*}
\mathcal{Y} = \left\{ v \in L^{2}_{\sigma}(\mathbb{R}^2) \big| \text{ there exist } v_F \in C^{\infty}_{\sigma, c}(\mathbb{R}^2) \text{ and } v_R\in\mathcal{R} \text{ such that } v|_{\mathcal{F}} =  v_F|_{\mathcal{F}} \text{ and }  v|_{\mathcal{S}_0} =  v_S|_{\mathcal{S}_0} \right\},
\end{equation*}
which is dense in $ \underline{\mathcal{V}} $. Therefore there exists a base $ \{w_i\}_{i\in\mathbb{N}}$ of the Hilbert space $ \underline{\mathcal{V}} $ such that $ w_i \in \mathcal{Y}$ for all $i$. We consider the approximate solution 
\begin{equation*}
\tilde{u}_N(t,x) = \tilde{u}_{N,R}(t,x) = \sum_{i=1}^N g_{i,N}(t)w_i(x),
\end{equation*}
where we forgot $ R $ for simplicity. The function $ \tilde{u}_N $ satisfies
\begin{equation}\label{gal:equ}
(\partial_t \tilde{u}_N,w_j)_{\mathcal{H}} =  2\nu a(\tilde{u}_N,w_j)+2\nu a(H,w_j)+ b_R(\tilde{u}_N,\tilde{u}_N,w_j) -b(H,w_j,\tilde{u}_N)-b(\tilde{u}_N,w_j,H), \quad 
\tilde{u}_N|_{t=0} = \tilde{u}_{N0}. 
\end{equation}
where $ \tilde{u}_{N0}$ is the orthogonal projection in $ \mathcal{H} $ of $\tilde{u}_0 $ onto the space spanned by $ w_1,\dots, w_N $. The existence of such $ g_{i,N} $ is due to the Cauchy-Lipschitz theorem applied to the system of ODE:
\begin{equation*}
\mathcal{G}_N' = \mathcal{M}^{-1}_N\left[2\nu\mathcal{A}_N\mathcal{G}_N+2\nu\mathcal{A}_{N,H}-\mathcal{B}_{N,H_1}(\mathcal{G}_N)-\mathcal{B}_{N,H_3}(\mathcal{G}_N)+\mathcal{B}_N(\mathcal{G}_N,\mathcal{G}_N) \right], \quad 
\mathcal{G}_N(0) = \mathcal{G}_{N,0},
\end{equation*} 
where 
\begin{gather*}
\mathcal{M}_N=[(w_i,w_j)_{\mathcal{H}}]_{1\leq i,j \leq N}, \quad \mathcal{G}_N = [g_{1,N} \dots g_{N,N}]^{T}, \quad \mathcal{A}_N = [a(w_i,w_j)]_{1\leq i, j \leq N}, \\
\left[\mathcal{B}_{N,H_1}(u)\right]_j=\sum_{k=1}^Nu_kb(H,w_j,w_k),  \quad \left[\mathcal{B}_{N,H_3}(u)\right]_j=\sum_{i=1}^N u_ib(w_i,w_j,H), \\
\left[ \mathcal{B}_N(u,v) \right]_j= \sum_{i,k = 1}^N u_i v_k b_R(w_i,w_j,w_k), \quad \left[\mathcal{A}_{N,H}\right]_j = a(H,w_j). 
\end{gather*} 
Note that $ \mathcal{M}_N $ is invertible because $ \{w_i\}_{i\in\mathbb{N}} $ are linear independent in $ \mathcal{H} $.

The Cauchy-Lipschitz theorem ensures a local in time existence for the functions $ g_{i,N}$. To prove that the existence is in all the interval $ [0,T] $ we need an estimate that leads us to conclude that the function $ g_{i,N} $ are defined in all $[0,T]$. To do that we multiply (\ref{gal:equ}) by $ g_{j,N}$ and we sum over $ j $ to obtain
\begin{equation}
\label{ene:equ}
\frac{1}{2}\frac{d}{dt}\|\tilde{u}_N\|_{\mathcal{H}}+2\nu\int_{\mathcal{F}_0}|D(\tilde{u}_N)|^2 dx+2\nu \alpha \int_{\partial \mathcal{S}_0}|\tilde{u}_N- \tilde{u}_{N,\mathcal{S}}|^2ds = 2\nu a(H,\tilde{u}_N)-b(\tilde{u}_N,\tilde{u}_N,H).
\end{equation}  
We now estimate the right hand side of the last equality. Note that for any $\varepsilon $ there exists $ C_{\varepsilon} $ such that
\begin{equation*}
|a(H,\tilde{u}_N)| \leq C_{\varepsilon}+\varepsilon\left(\int_{\mathcal{F}_0}|D(\tilde{u}_N)|^2 dx+ \int_{\partial \mathcal{S}_0}|\tilde{u}_N- \tilde{u}_{N,\mathcal{S}}|^2ds \right) \quad \text{ and that } \quad |b(\tilde{u}_N,\tilde{u}_N,H)| \leq C\|\tilde{u}_N\|_{\mathcal{H}}^2,
\end{equation*}
where $ C $ and $  C_{\varepsilon} $ do not depend on $ N$ and $ R $. If we integrate (\ref{ene:equ}) in $ (0,t) $, we use the two inequality above and we bring on the left the terms multiplied by $ \varepsilon $  we get
\begin{equation*}
\|\tilde{u}_N\|_{\mathcal{H}}^2+\int_0^t\int_{\mathcal{F}_0}|D(\tilde{u}_N)|^2 dx+ \int_0^t\int_{\partial \mathcal{S}_0}|\tilde{u}_N- \tilde{u}_{N,\mathcal{S}}|^2ds 
\leq  \int_0^t  \left(C+\|\tilde{u}_N\|_{\mathcal{H}}^2\right)dt +\|\tilde{u}_{N0}\|_{\mathcal{H}}^2.
\end{equation*}
Using the Gr\"onwall lemma we obtain the estimate
\begin{equation*}
\|\tilde{u}_N\|_{\mathcal{H}}^2 \leq te^{tC}\left(C\frac{t}{2}+\|\tilde{u}_{N0}\|_{\mathcal{H}}^2 \right)+Ct+\|\tilde{u}_{N0}\|_{\mathcal{H}}^2,
\end{equation*}
which leads us to conclude that the function $g_{i,N} $ can be extended in all $[0,T]$.

Moreover, by the fact that $ \|\tilde{u}_{N0}\|_{\mathcal{H}} \leq \|\tilde{u}_{0}\|_{\mathcal{H}} $ and by the Korn inequality, we conclude that
\begin{gather*}
\tilde{u}_{N} \in L^{\infty}((0,T);\mathcal{H})\\
\tilde{u}_{N} \in L^2((0,T);\underline{\mathcal{V}})
\end{gather*} 
are uniformly bounded in both the spaces. This leads us to conclude that there exists $ \tilde{u} \in L^{\infty}((0,T);\mathcal{H})\cap L^2((0,T);\underline{\mathcal{V}}) $ such that $\tilde{u}_N $ converges to $ \tilde{u}$ weakly in $ L^2((0,T);\underline{\mathcal{V}}) $ and *-weakly in $ L^{\infty}((0,T);\mathcal{H}) $ as $ N $ goes to $ +\infty$.

We pass to the limit in (\ref{gal:equ}). The only not triviality is to prove the convergence of the non-linear term, i.e. $ b_R(\tilde{u}_N,\tilde{u}_N,w_j)$ converges to $ b_R(\tilde{u},\tilde{u},w_j)$. The idea is to notice that $ \tilde{u}_N$ is relatively compact in $ L^2((0,T);L^2_{loc}(\mathbb{R}^2)) $, in fact this follows from the proof of Theorem 1 in \cite{PS}, where the only difference is the estimate
\begin{equation*}
\|f_N\|_{\underline{\mathcal{V}}'} \leq C(1+\|\tilde{u}_N\|_{\underline{\mathcal{V}}}+\|\tilde{u}_N\|_{\underline{\mathcal{V}}}^2),
\end{equation*}    
with $f_N $ defined by 
\begin{equation*}
\langle f_N, w \rangle = 2\nu a(\tilde{u}_N, w)+2\nu a(H,w)+ b_R(\tilde{u}_N,\tilde{u}_N,w)-b(H,w,\tilde{u}_N)-b(\tilde{u}_N,w,H).
\end{equation*}

At this point we are able to pass to the limit in 
\begin{align*}
(\tilde{u}_N(t),\varphi(t))_{\mathcal{H}}-(\tilde{u}_{N0},\varphi(0))_{\mathcal{H}}= \int_0^t \Big[&(\tilde{u}_N,\partial_t \varphi)_{\mathcal{H}}+2\nu a(\tilde{u}_N,\varphi)+2\nu a(H,\varphi) \\  -b_R(\tilde{u}_N&,\varphi,\tilde{u}_N)-b(H,\varphi,\tilde{u}_N)-b(\tilde{u}_N,\varphi, H) \Big]dt.
\end{align*}
which means that $ \tilde{u} = \tilde{u}_R $ satisfies
\begin{align*}
(\tilde{u}_R(t),\varphi(t))_{\mathcal{H}}-(\tilde{u}_{R0},\varphi(0))_{\mathcal{H}}= \int_0^t \Big[&(\tilde{u}_R,\partial_t \varphi)_{\mathcal{H}}+2\nu a(\tilde{u}_R,\varphi)+2\nu a(H,\varphi) \\  -b_R(\tilde{u}_R&,\varphi,\tilde{u}_R)-b(H,\varphi,\tilde{u}_R)-b(\tilde{u}_R,\varphi, H) \Big]dt.
\end{align*}

\bigskip

{\bf Limit of the solutions of the truncated problems.} We note that the energy estimate do not depend on $ R $, so there exists sequence $ \tilde{u}_{k,R_k} $ converging to $ \tilde{u} \in C((0,T);\mathcal{H})\cap L^2((0,T);\underline{\mathcal{V}}) $ *-weakly in $ L^{\infty}((0,T);\mathcal{H}) $ and weakly in $ L^2((0,T);\underline{\mathcal{V}}) $  as $ k $ goes to $ +\infty$.

This convergences do not leads us to pass directly to the limit because of the non-linearity of $ b_R $, in other words we have to find an argument to prove that 
\begin{equation*}
\int_{0}^tb_{R_k}(\tilde{u}_{k,R_k},\tilde{u}_{k,R_k},\varphi) \to \int_{0}^tb(\tilde{u},\tilde{u},\varphi)dx, \quad \text{ as } k \text{ goes to } +\infty.
\end{equation*}  
As presented in the paper \cite{PS}, it is enough to prove that $ \tilde{u}_{k,R_k} $ is relatively compact in $ L^2((0,T);L^2_{loc}(\mathbb{R}^2))$. We have already presented this compactness property for $ \tilde{u}_{N,R} $, but the estimates are $ R $ depending so we cannot directly conclude.

The idea is to apply the Aubin-Lions lemma to get the compactness result. First of all we note that $ \tilde{u}_{k,R_k} $ are uniformly bounded in $ L^4(0,T;L^4(\mathcal{F}_0)) $, in fact
\begin{align*}
\|\tilde{u}_{k,R_k}\|_{L^4(0,T;L^4(\mathcal{F}_0))}^4 \leq & \int_0^t \|\tilde{u}_{k,R_k}\|_{L^4(\mathcal{F}_0)}^4 dt \\  \leq & \int_0^t\|\tilde{u}_{k,R_k}\|^2_{L^2(\mathcal{F}_0)}\|\nabla \tilde{u}_{k,R_k}\|^2_{L^2(\mathcal{F}_0)} dt \\ \leq & \|\tilde{u}_{k,R_k}\|_{L^{\infty}(0,T;L^2(\mathcal{F}_0))} \|\nabla \tilde{u}_{k,R_k}\|_{L^2(0,T;L^2(\mathcal{F}_0))}.
\end{align*} 
This leads us to prove that $ \partial_t \tilde{u}_{k,R_k} $ is uniformly bounded in $ L^2((0,T);\mathcal{V}') $, in fact the only non-linear term that can be an issue is 
\begin{equation*}
\int_{\mathcal{F}_0} [(\tilde{u}_{k,R_k}\cdot\nabla)g]\cdot \tilde{u}_{k,R_k} dx, 
\end{equation*}
where $ g \in L^2(0,T; \mathcal{V})$, but it can be bound by
\begin{align*}
\left|\int_0^T \int_{\mathcal{F}_0} [(\tilde{u}_{k,R_k}\cdot\nabla)g]\cdot \tilde{u}_{k,R_k} dx dt \right| \leq & C \int_0^T \|\tilde{u}_{k,R_k}\|_{L^4(\mathcal{F}_0)}^2\|\nabla g\|_{L^2(\mathcal{F}_0)} dt \\
\leq & C\|\tilde{u}_{k,R_k}\|_{L^4(0,T;L^4(\mathcal{F}_0))}^2\|g\|_{L^2(0,T; \mathcal{V})}.
\end{align*}
It is clear from Aubin-Lions lemma that $ \{\tilde{u}_{k,R_k}|_{B_r(0)}\}_{k\in\mathbb{N}}$ is relatively compact in $ L^2(0,T;L^2(B_r(0))) $, for every ball $ B_r (0) $ of radius $ r \in \mathbb{N} $ such that $ \mathcal{S}_0 \subset B_r(0) $. By extracting a diagonal subsequence we get that 
$ \{\tilde{u}_{k,R_k}\}_{k\in\mathbb{N}}$ is relatively compact in $ L^2(0,T;L_{loc}^2(\mathbb{R}^2)) $.   

We can now pass to the limit in the weak formulation to get the desired result.
%

\bigskip

{\bf Improved regularity for $(l,r)$.} In two dimensions the  Kirchhoff potentials  are the solutions $\Phi=(\Phi_{i})_{i=1,2, 3}$ of the following problems:
\begin{equation*} 
-\Delta \Phi_i = 0 \quad   \text{for}  \ x\in \mathcal{F}_{0}   ,
\end{equation*}
\begin{equation*}
\Phi_i \longrightarrow 0 \quad  \text{for}  \ |x| \rightarrow  \infty,
\end{equation*}
\begin{equation*} 
\frac{\partial \Phi_i}{\partial n}=K_i
\quad  \text{for}  \  x\in \partial \mathcal{F}_{0}   ,
\end{equation*}
where
\begin{equation*} 
(K_{1},\, K_{2}, \, K_{3}) =(n_1,\, n_2 ,\, x^\perp \cdot n).
\end{equation*}
These functions are smooth and decay at infinity as follows:
\begin{equation} \nonumber
\nabla \Phi_i = {\mathcal O}\left( \frac{1}{|x|^{2}}\right) \ \text{ and } \nabla^{2}   \Phi_i  = {\mathcal O}\left( \frac{1}{|x|^{3}}\right) \  \text{ as } x \rightarrow \infty .
\end{equation}

We now define three functions $v_i $,  for $i = 1,2,3$, defined by
\begin{equation*}
v_i = \nabla  \Phi_i \text{ in } \mathcal{F}_{0}  \text{ and }v_i =
\left\{\begin{array}{ll}
e_i & \text{if} \ i=1,2 ,\\ \relax
 x^\perp  & \text{if} \ i=3,
\end{array}\right.
  \text{ in } \mathcal{S}_{0} ,
\end{equation*}
and which are  $  \widehat{\mathcal{V}}$.
The body's equations can then be rephrased as follows:
\begin{equation*}
    \mathcal{M}  \begin{bmatrix} \ell \\ r \end{bmatrix} '
   = ( 2\nu a(\tilde{u}, v_i )+2\nu a(H,v_i) + b(\tilde{u}, \tilde{u}, v_i)+b(H,\tilde{u}, v_i)-b(\tilde{u},v_i,H)  )_{i \in \{1,\ldots,3\}}  ,
\end{equation*}
where 
\begin{equation*}
\mathcal{M}=
\begin{bmatrix} m \Id_2 & 0 \\ 0 & \mathcal{J} \end{bmatrix}
+ \begin{bmatrix} \displaystyle\int_{\mathcal{F}_0} \nabla \Phi_a \cdot \nabla \Phi_b \, dx \end{bmatrix}_{a,b \in \{1,2,3\}}  .
\end{equation*}
Since the matrix  $\mathcal{M}$ is symmetric and positive definite, applying Proposition \ref{ext:ab}  yields that $(\ell ,r)$ is in $H^1 (0,T;   \R^2 \times \R )  $.
\end{proof}

\ \par \ 

 {\bf Acknowledgements.} 
The author was supported by the Agence Nationale de la Recherche, 
 Project IFSMACS, grant ANR-15-CE40-0010 and the Conseil R\'egional dÄ'Aquitaine, grant 2015.1047.CP.


\end{document}